\documentclass[10pt]{article}

\usepackage{amsmath, amssymb, amsthm}
\usepackage{fullpage}
\usepackage{float}
\usepackage[all]{xy}
\usepackage{graphicx, epstopdf}

\newtheorem{definition}{Definition}[section]
\newtheorem{theorem}[definition]{Theorem}

\newtheorem{lemma}[definition]{Lemma}
\newtheorem{corollary}[definition]{Corollary}
\newtheorem{proposition}[definition]{Proposition}

\DeclareMathOperator{\GL}{GL}
\DeclareMathOperator{\Sp}{Sp}
\DeclareMathOperator{\rk}{rk}
\DeclareMathOperator{\GSp}{GSp}
\DeclareMathOperator{\ord}{ord}
\DeclareMathOperator{\tr}{tr}
\DeclareMathOperator{\gen}{gen}

\newlength\tindent
\usepackage{titlesec}
\titleformat*{\section}{\normalsize\centering\scshape}  

\title{\centering\large\scshape Fourier coefficients of degree two Siegel-Eisenstein series with trivial character at squarefree level}
\author{\centering\normalsize\scshape Martin J. Dickson}
\date{}

\begin{document}
\maketitle

\begin{quote} \textsc{Abstract}.  We compute the Fourier coefficients of a basis of the space of degree two Siegel-Eisenstein series of square-free level $N$ transforming with the trivial character.  We then apply use these formul\ae\ to present some explicit examples of higher representation numbers attached to non-unimodular quadratic forms.\end{quote}

\begin{quote} \textsc{Keywords}.  Siegel modular forms; Eisenstein series; Fourier coefficients; quadratic forms; theta series.\end{quote}

\begin{quote} \textsc{Mathematics Subject Classification}.  11F30 \end{quote}

\section{Introduction}

\noindent Eisenstein series have played an important r\^{o}le in the theory of Siegel modular forms, dating back to Siegel's Hauptsatz which expresses the genus-average Siegel theta series as a linear combination of Eisenstein series.  Despite their distinguished history, there are still some perhaps surprising gaps in our knowledge of the Eisenstein series.  An important example of this is that the Fourier coefficients of the Siegel-Eisenstein series remain unknown in many cases.\\

Various authors have worked on the case Siegel-Eisenstein series for the full symplectic group $\Sp_{2n}(\mathbb{Z})$  and we mention only a sample of the results here.  Maass (\cite{Maass1964}, \cite{Maass1972}) obtained a formula for the Fourier coefficients of degree $2$ Siegel-Eisenstein series by explicitly computing the local densities in Siegel's formula.  His results were also obtained by Eichler--Zagier (\cite{EichlerZagier1985}) by realising the Siegel-Eisenstein series as the Maass lift of the corresponding Jacobi-Eisenstein series.  More recently, formul\ae\ for general degree $n$ have been obtained via different methods by Katsurada (first degree $3$ \cite{Katsurada1997}, then any degree \cite{Katsurada1999}) and Kohnen (even degree \cite{Kohnen2002}), Choie-Kohnen (odd degree \cite{ChoieKohnen2011}).  Katsurada's approach has more in common with Maass' approach, whereas Kohnen's approach (for even degree) is at least nominally closer to that of Eichler--Zagier with his linearised version of the Ikeda lift playing a r\^{o}le simlar to the Maass lift.\\

In the case of modular forms transforming with character $\chi$ under the Hecke-type congruence subgroup $\Gamma_0^{(n)}(N)$ the situation is less well-known.  To the author's knowledge, the only explicit results in the literature pertain to a single Eisenstein series when the degree is $n=2$ and the character $\chi$ is primitive modulo $N$.  First Mizuno (\cite{Mizuno2009}) considered the case of squarefree $N$ and obtained the Fourier coefficients by realising a level $N$ Eisenstein series as a Maass lift of a corresponding level $N$ Jacobi-Eisenstein series.  The argument is more difficult than Eichler-Zagier's at level $1$ since the author requires analytic machinery to prove the coincidence of the lift with the desired form.  An extension of this result, dropping the assumption that $N$ be squarefree, was obtained by Takemori (\cite{Takemori2012}) by computing the local densities.  A related result, again in the context of $n=2$ and arbitrary level $N$ but now with no restrictions on the character, is due to Yang, who explicitly computes the local densities in Siegel's theorem.  Yang's methods are rather different: following Weil and Kudla he interprets those local densities in terms of the local Whittaker function coming from the representation attached to the Eisenstein series.  He explicitly computes this latter quantity (\cite{Yang1998} for $p \neq 2$; \cite{Yang2004} for $p=2$), which is essentially equivalent to computing the Fourier coefficients of the genus-average theta series, hence the Fourier coefficients of a Siegel-Eisenstein series.\\

Since quadratic forms are rarely unimodular the corresponding theta series will usually be modular forms of level $N > 1$.  Additionally, it is important for the arithmetic theory of these quadratic forms that we have explicit formul\ae\ for Fourier coefficients for a \textit{basis} of the space of Siegel-Eisenstein series.  Although we do have some Fourier coefficients when $N>1$, we only have these for a single Eisenstein series, but (for large enough weight) the dimension of the space of Siegel-Eisenstein series of level $N$ is strictly bigger than $1$.\\

In this paper we consider only the case $n=2$.  After setting up notation in \S\ref{notation}, in \S\ref{sctn:calculation-of-fcs} we consider the case when the level $N$ is squarefree and the character $\chi$ is the trivial character modulo $N$.  Under these assumptions we can use the following very simple method to obtain relations amongst the Fourier coefficients: let $p$ be a prime not dividing $N$, then an appropriate sum of level $Np$ Eisenstein series produces a level $N$ Eisenstein series, and after acting on this relation with Hecke operators at $p$ the explicit formul\ae\ from \cite{Walling2012} produce enough linear relations among the Fourier coefficients to write down a formula for the coefficients of a level $Np$ Eisenstein series in terms of those of the level $N$ one.  Using a convenient level $1$ formula, namely that of \cite{EichlerZagier1985}, one can then argue by induction to obtain Fourier coefficients for a full basis of the Eisenstein subspace in the case of squarefree level and trivial character.  This is carried out in Lemma \ref{basic-formula-for-ai}: the main bulk of the computation is then placing these in a more elucidating form as stated in Theorem \ref{fc-of-eisenstein-series}.\\

Let us remark that the important feature that makes this work is that the level $Np$ Fourier coefficients add up to something known in the base case of the induction.  So for example in the case of primitive character this approach seems unlikely to succeed.  In the case of squarefree level the transformation character will always be a product of primitive and trivial characters, and if one knows the Fourier coefficients for Eisenstein series transforming with a given primitive character $\chi$ then one can argue as suggested above to obtain Fourier coefficients of Eisenstein series of any squarefree level and character which has $\chi$ as the underlying primitive character.  However one would need to know the Fourier coefficients for a full basis at the primitive stage in order to deduce the Fourier coefficients for a full basis at later stages.  As noted above no such formul\ae\ for a full basis are currently available.\\

Finally in \S\ref{sctn:representation-numbers} we emphasise this point regarding the importance of having Fourier coefficients for a full basis of the Eisenstein subspace by showing how one can compute the genus representation numbers of an integral quadratic form by combining knowledge of the Fourier coefficients of a basis for the Eisenstein subspace with the well-known formul\ae\ for the value a theta series takes at a $0$-dimensional cusp of (the Satake compactification of) $\Gamma_0^{(2)}(N) \backslash \mathbb{H}_2$.  There is a finite number of integral quadratic forms which have a single-class genus, and from the viewpoint of degree $2$ representation numbers only the $8$-dimensional ones have dimension large enough to study via Siegel-Eisenstein series (i.e. the Eisenstein series of degree $2$ and weight $4$ converges) of even weight (since odd weight Eisenstein series are problematic to define with trivial character).  Amongst these 8-dimensional integral quadratic forms, or equivalently even integral lattices, only 5 satisfy the condition that their level be squarefree and their character trivial.  Of course one of these is the unimodular lattice $E_8$ for which degree $2$ representation numbers (i.e. explicit formul\ae\ for the number of times it represents a quadratic form in $2$ variables) follow (for example) from the formula of \cite{EichlerZagier1985}.  The remaining $4$ have small prime level and for these we will note how the methods of this paper give new closed formul\ae\ for their degree $2$ representation numbers.\\

\textbf{Acknowledgements.}  The author would like to thank his supervisor Dr. L. Walling for suggesting this problem and her help with related questions about Siegel modular forms.  He would also like the thank their referee for their careful reading of their manuscript and suggestions.  The author's research is supported by an EPSRC Doctoral Training Grant.

\section{Siegel-Eisenstein series and Hecke operators}\label{notation}

\textbf{Preliminaries.}  For any ring $R$, we let $R^{n \times n}$ denote the set of $n \times n$ matrices over $R$, and $R^{n \times n}_{\text{sym}}$ the additive subgroup of symmetric matrices.  Define the algebraic group
\[\GSp_{2n} = \left\{g \in \GL_{2n}; {}^tg \left(\begin{smallmatrix} & -1_2 \\ 1_2 &  \end{smallmatrix}\right) g = \lambda(g)\left(\begin{smallmatrix} & -1_2 \\ 1_2 &  \end{smallmatrix}\right)\text{ for some }\lambda(g) \in \GL_1\right\}. \]
If $R$ is a subring of $\mathbb{R}$ we write $\GSp_{2n}^+(R)$ for the subgroup of $\GSp_{2n}(R)$ comprised of those $g$ for which $\lambda(g) > 0$.  $\lambda : \GSp_{2n} \to \GL_1$ defines a homomorphism, the kernel is by definition $\Sp_{2n}$.  We define the congruence subgroup
\[\Gamma^{(n)}_0(N) = \left\{\left(\begin{matrix} A & B \\ C & D \end{matrix}\right) \in \Sp_{2n}(\mathbb{Z});\: C \equiv 0 \bmod N\right\}.\]
We also write $\Gamma^{(n)}$ for $\Gamma^{(n)}_0(1) = \Sp_{2n}(\mathbb{Z})$.  Let $\mathbb{H}_n = \{Z \in \mathbb{C}^{n \times n}_{\text{sym}}; \Im(z)>0\}$ be the Siegel upper half space of degree $n$.  Let $k$ be a positive integer; we define the weight $k$ slash operator on functions $f : \mathbb{H}_n \to \mathbb{C}$ for $g = \left(\begin{smallmatrix} A & B \\ C & D \end{smallmatrix}\right) \in \GSp^+_{2n}(\mathbb{R})$ by
\begin{equation}\label{slash-operator-normalisation} (f|_k g)(Z) = \lambda(g)^{nk/2} \det(CZ + D)^{-k} f((AZ+B)(CZ+D)^{-1}).\end{equation}
We say a holomorphic function $f : \mathbb{H}_n \to \mathbb{C}$ is a Siegel modular form of degree $n$, weight $k$, level $N$, and character $\chi$ (modulo $N$) if 
\[f |_k \gamma = \chi(\det(D))f\] 
for all $\gamma = \left(\begin{smallmatrix} A & B \\ C & D \end{smallmatrix}\right) \in \Gamma^{(n)}_0(N)$.  The complex vector space of such functions is denoted $\mathcal{M}^{(n)}_k(N, \chi)$.  We are mainly interested in the case when $\chi$ is the trivial character $\mathbf{1}_N$ modulo $N$, for which we abbreviate $\mathcal{M}_k^{(n)}(N) := \mathcal{M}_k^{(n)}(N, \mathbf{1}_N)$.  In this paper we will mostly work with the case of Siegel degree two in which case we drop the superscript ${}^{(2)}$, so for example $\mathcal{M}_k(N) := \mathcal{M}_k^{(2)}(N)$, $\Gamma_0(N) := \Gamma_0^{(2)}(N)$, and $\Gamma := \Gamma^{(2)} = \Sp_4(\mathbb{Z})$. \\

\textbf{Siegel-Eisenstein series.} Fix a positive integer $N$ and assume $k \geq 4$ is even.  We define an Eisenstein series for each $0$-cusp of the Satake compactification $\mathcal{S}(\Gamma_0(N) \backslash \mathbb{H}_2)$ of $\Gamma_0(N) \backslash \mathbb{H}_2$; that is, for each element of the double coset $\Gamma_\infty \backslash \Gamma/ \Gamma_0(N)$, where
\[\Gamma_\infty = \left\{\left(\begin{matrix} A & B \\ 0 & D \end{matrix}\right) \in \Gamma\right\}.\]
To do so, pick such an element $\Gamma_\infty \gamma_0 \Gamma_0(N)$, and define
\[ \mathbb{E}_{\gamma_0}(\tau) = \sum_\gamma \det(C_{\gamma}\tau + D_{\gamma})^{-k}\]
where $\Gamma_\infty \gamma_0 \Gamma_0(N) = \bigsqcup_\gamma \Gamma_\infty \gamma$, and $\gamma = \left(\begin{smallmatrix} A_\gamma & B_\gamma \\ C_\gamma & D_\gamma \end{smallmatrix}\right)$.  One may easily check that, under the restriction that $k$ be even, this series depends only on the double coset $\Gamma_\infty \gamma_0 \Gamma_0(N)$ and is independent of the choice of the representative $\gamma$.  Under the assumption $k \geq 4$ (using the Hecke trick for small weights) the series converges and thus defines a nonzero element of $\mathcal{M}_k(N)$.  Letting $\gamma_0$ vary over a system of representative for $\Gamma_\infty \backslash \Gamma / \Gamma_0(N)$ we obtain a basis of the Siegel-Eisenstein subspace of $\mathcal{M}_k(N)$.  This basis is characterised by the property that $\mathbb{E}_{\gamma_0}$ is the unique Siegel-Eisenstein series that takes value $1$ at the cusp corresponding to $\gamma_0$ and $0$ at all other cusps, as can easily be checked from the definition.  We shall refer to this basis as the \textit{natural basis} for the subspace of $\mathcal{M}_k(N)$ spanned by the Siegel-Eisenstein series.  As we shall see this is not an eigenbasis for Hecke operators at primes dividing the level, but is useful for other computational purposes.\\

In order to describe the action of Hecke operators on Siegel-Eisenstein series it is convenient to fix a choice of the representative $\gamma_0$.  We do so in the same fashion as \cite{Walling2012}, recalling the discussion via Lemmas \ref{eisenstein-quotient-in-csp} and \ref{csp-orbit-rank-characterisation}.  They are both readily verified; the details of the computations are in \cite{Walling2012} \S2.

\begin{lemma}\label{eisenstein-quotient-in-csp}  Let $\mathcal{S}$ denote the set of coprime symmetric pairs of $2 \times 2$ matrices, and let $G \in \GL_2(\mathbb{Z})$ act on $\mathcal{S}$ by $(C, D) \mapsto (GC, GD)$.  \begin{enumerate}  \item The map
\[\left(\begin{matrix} A & B \\ C & D \end{matrix}\right) \mapsto \left(C, D\right)\]
induces a bijection between $\Gamma_\infty \backslash \Gamma$ and $\GL_2(\mathbb{Z}) \backslash \mathcal{S}$.  
\item  Let $\left(\begin{smallmatrix} A_\gamma & B_\gamma \\ C_\gamma & D_\gamma \end{smallmatrix}\right) \in \Gamma_0(N)$ act (from the right) on $\GL_2(\mathbb{Z}) \backslash \mathcal{S}$ by 
\[\GL_2(\mathbb{Z})(C, D) \mapsto \GL_2(\mathbb{Z})(CA_\gamma + DC_\gamma, CB_\gamma + DD_\gamma).\] 
Then the bijection of (1.) is an isomorphism of $\Gamma_0(N)$-sets. \end{enumerate}\end{lemma}

Let $\rho = (N_0, N_1, N_2)$ be a multiplicative partition of $N$, i.e. a triple of positive integers such that $N_0 N_1 N_2 = N$ (since $N$ is squarefree the $N_i$ are coprime).  Pick a symmetric matrix $M_{\rho}$ such that
\[M_{\rho} \equiv \begin{cases} \left(\begin{smallmatrix} 0 & 0 \\ 0 & 0 \end{smallmatrix}\right) \bmod N_0, \\ \left(\begin{smallmatrix} 1 & 0 \\ 0 & 0 \end{smallmatrix}\right) \bmod N_1, \\ \left(\begin{smallmatrix} 1 & 0 \\ 0 & 1 \end{smallmatrix}\right) \bmod N_2.\end{cases}\]

For $C$ a square integer matrix, let $\rk_q(C)$ denote the rank of the reduced matrix over $\mathbb{Z}/q\mathbb{Z}$.  We can use this data to characterise the $\Gamma_0(N)$ orbits of Lemma \ref{eisenstein-quotient-in-csp}:

\begin{lemma}\label{csp-orbit-rank-characterisation}  Let $(C, D)$ be a coprime symmetric pair.  Then $\GL_2(\mathbb{Z})(C, D)$ is in the $\Gamma_0(N)$-orbit of $\GL_2(\mathbb{Z})(M_\rho, I)$ if and only if $\rk_q(C) = \rk_q(M_\rho)$ for all $q \mid N$.  \end{lemma}

\begin{corollary}  Let $(N_0, N_1, N_2)$ be a multiplicative partition of $N$, and let $p$ be a prime not dividing $N$.  Then
\[\begin{aligned} \GL_2(\mathbb{Z})(M_{(N_0, N_1, N_2)}, I)\Gamma_0(N) &= \GL_2(\mathbb{Z})(M_{(pN_0, N_1, N_2)}, I)\Gamma_0(Np) \\
&\qquad \cup \GL_2(\mathbb{Z})(M_{(N_0, pN_1, N_2)}, I)\Gamma_0(Np) \\
&\qquad \cup \GL_2(\mathbb{Z})(M_{(N_0, N_1, pN_2)}, I)\Gamma_0(Np). \end{aligned}\] \end{corollary}
\begin{proof}  Follows immediately from Lemma \ref{csp-orbit-rank-characterisation}.  \end{proof}

\noindent Let $\Gamma_\infty \gamma_0 \Gamma_0(N)$ be a cusp for $\Gamma_0(N)$.  Since the summands for $\mathbb{E}_{\gamma_0}$ depend only on the bottom (block) row, Lemma \ref{eisenstein-quotient-in-csp} allows us to express $\mathbb{E}_{\gamma_0}$ as a sum over the $\Gamma_0(N)$ orbit of $\GL_2(\mathbb{Z})(C, D)$ for some coprime symmetric pair $(C, D)$.  By Lemma \ref{csp-orbit-rank-characterisation} we may assume that this is the $\Gamma_0(N)$ orbit of $\GL_2(\mathbb{Z})(M_\rho, I)$ for some $M_\rho$ as defined above.  We therefore take our choice of representative $\gamma_0$ to be $\left(\begin{smallmatrix} I & 0 \\ M_\rho & I\end{smallmatrix}\right)$, and hence identify a cusp $\Gamma_\infty \gamma_0 \Gamma_0(N)$ with a matrix $M_\rho$.  $M_\rho$ is in turn identified with a triple $(N_0, N_1, N_2)$ such that $N_0N_1N_2=N$ by letting $N_i$ be the product of all primes $q \mid n$ at which $M_\rho$ has rank $i$.  For a triple $(N_0, N_1, N_2)$ thus corresponding to a cusp $\Gamma_\infty \gamma_0 \Gamma_0(N)$ we define $\mathbb{E}_{(N_0, N_1, N_2)} = \mathbb{E}_{\gamma_0}$.

\begin{corollary}  Let $N$ be square-free, $(N_0, N_1, N_2)$ such that $N_0N_1N_2 = N$, and let $p$ be a prime not dividing $N$.  Then
\[\mathbb{E}_{(N_0, N_1, N_2)} = \mathbb{E}_{(pN_0, N_1, N_2)} + \mathbb{E}_{(N_0, pN_1, N_2)} + \mathbb{E}_{(N_0, N_1, pN_2)}.\] \end{corollary}
\begin{proof}  Follows immediately from Corollary 2.3.  \end{proof}

\noindent \textbf{Hecke operators.} Let $p$ be a prime not dividing $N$.  Define the Hecke operator $T(p)$ by
\[f | T(p) = p^{k-3} \sum_{i} f |_k \gamma_i\]
where 
\[\Gamma_0(N) \left(\begin{matrix}1_n & \\ & p 1_n \end{matrix}\right)\Gamma_0(N) = \bigsqcup_i \Gamma_0(N)\gamma_i.\]
Similarly, define another Hecke operator $T_1(p^2)$ by 
\[f | T_1(p^2) = p^{k-3} \sum_{i} f |_k \gamma_i\]
where 
\[\Gamma_0(N) \left(\begin{smallmatrix} 1 & & & \\ & p & & \\ & & p^2 & \\ & & & p \end{smallmatrix}\right) \Gamma_0(N) = \bigsqcup \Gamma_0(N) \gamma_i.\]
When $p \mid N$ we define the Hecke operators in exactly the same way, but we use the alternative notation $U(p)$, $U_1(p^2)$ to emphasise that $p \mid N$ affects the representatives in the coset decompositions.  Note that \cite{Walling2012} uses a definition of Hecke operators that is equivalent to our double coset definition except that the representative matrices differ by a factor of $p$.  This makes no difference because under the normalisation of the slash operator used by us (in (\ref{slash-operator-normalisation})) and \cite{Walling2012} scalars act trivially.  We now quote the results of \cite{Walling2012} in the case of trivial character:

\begin{proposition}\label{Up-action}  The action of the Hecke operators $U(p)$ on the level $Np$ Eisenstein series transforming with trivial character are as follows:
\[ \begin{aligned} \mathbb{E}_{(pN_0, N_1, N_2)} | U(p) &= \mathbb{E}_{(pN_0, N_1, N_2)} + (1-p^{-1})\mathbb{E}_{(N_0, pN_1, N_2)} + (1-p^{-1})\mathbb{E}_{(N_0, N_1, pN_2)}, \\
\mathbb{E}_{(N_0, pN_1, N_2)} | U(p) &= p^{k-1}\mathbb{E}_{(N_0, pN_1, N_2)} + (p^{k-1}-p^{k-3})\mathbb{E}_{(N_0, N_1, pN_2)}, \\
\mathbb{E}_{(N_0, N_1, pN_2)} | U(p) &= p^{2k-3} \mathbb{E}_{(N_0, N_1, pN_2)}. \end{aligned} \] \end{proposition}
\begin{proof}  These are special cases of \cite{Walling2012} Propositions 3.5, 3.6, and 3.7.  \end{proof}

\begin{proposition}\label{U1p2-action}  The action of the Hecke operators $U_1(p^2)$ on the level $Np$ Eisenstein series transforming with trivial character are as follows:
\[ \begin{aligned} \mathbb{E}_{(pN_0, N_1, N_2)} | U_1(p^2) &= (p+1)\mathbb{E}_{(pN_0, N_1, N_2)} + (p^{k-1} + 1)(1-p^{-1}) \mathbb{E}_{(N_0, pN_1, N_2)} + (1-p^{-2})\mathbb{E}_{(N_0, N_1, pN_2)}, \\
\mathbb{E}_{(N_0, pN_1, N_2)} | U_1(p^2) &= (p^{2k-2} + p)\mathbb{E}_{(N_0, pN_1, N_2)} + (p^{k-2} + 1)(p-p^{-1})\mathbb{E}_{(N_0, N_1, pN_2)},\\
\mathbb{E}_{(N_0, N_1, pN_2)} | U_1(p^2) &= (p^{2k-2} + p^{2k-3})\mathbb{E}_{(N_0, N_1, pN_2)}. \end{aligned} \] \end{proposition}
\begin{proof}  These are special cases of \cite{Walling2012} Propositions 3.8, 3.9, and 3.10.  \end{proof}

\begin{proposition}\label{Tp-action}  The action of the Hecke operator $T(p)$ on the level $N$ (where $p \nmid N$) Eisenstein series transforming with trivial character is:
\[\mathbb{E}_{(N_0, N_1, N_2)} | T(p) = (p^{2k-3} + p^{k-1} + p^{k-2} + 1)\mathbb{E}_{(N_0, N_1, N_2)}.\] \end{proposition}
\begin{proof}  This is a special case of \cite{Walling2012} Proposition 3.3.  \end{proof}

Let $f \in \mathcal{M}_k(N)$.  It is well-known that $f$ has a Fourier expansion supported on $2$-rowed, half-integral, positive semi-definite, symmetric matrices $T$.  We use the notation
\begin{equation}\label{fourier-expansion} f(Z) = \sum_T a(T; f) e(\tr(TZ))\end{equation}
where $e(z) = e^{2 \pi i z}$.  We will use the following result from \cite{HafnerWalling2002} for the action of the Hecke operators $U(p)$, $U_1(p^2)$, and $T(p)$ on Fourier expansions:

\begin{proposition}\label{action-on-fc}  Let $f(Z) = \sum_T a(T; f) e(\tr(TZ)) \in \mathcal{M}_k(N)$.  For $M$ any matrix, write $T[M] = {}^tMTM$.  Then, for $p \mid N$,
\[\begin{aligned} a(T; f|U(p)) &= a(pT; f), \\
a(T; f|U_1(p^2)) &= \sum_{\alpha \bmod p} a\left(T\left[\left(\begin{matrix}1 & 0 \\ \alpha & p \end{matrix}\right)\right]; f\right) + a\left(T\left[\left(\begin{matrix}p & 0 \\ 0 & 1 \end{matrix}\right)\right]; f\right), \end{aligned}\]
and, for $p \nmid N$,
\[a(T; f|T(p)) = a(pT; f) +  p^{k-2} \left(\sum_{\alpha \bmod p} a\left(\frac{1}{p}T\left[\left(\begin{matrix}1 & 0 \\ \alpha & p \end{matrix}\right)\right]; f\right) + a\left(\frac{1}{p} T\left[\left(\begin{matrix}p & 0 \\ 0 & 1 \end{matrix}\right)\right]; f\right)\right) + p^{2k-3}a\left(\frac{1}{p}T; f\right).\] \end{proposition}

\section{Calculation of the Fourier coefficients}\label{sctn:calculation-of-fcs}

\noindent \textbf{A computation based on \cite{Walling2012}.}  Fix a partition $(N_0, N_1, N_2)$ of the squarefree integer $N$ and let $\mathbb{E}_{(N_0, N_1, N_2)}$ be the associated Eisenstein series transforming with the trivial character modulo $N$.  Recall that (for $T$ a $2$-rowed, half-integral, positive semi-definite, symmetric matrix) we wrote $a(T; f)$ for the $T$th Fourier coefficient of $f$.  We also write $a(T) = a(T; \mathbb{E}_{(N_0, N_1, N_2)})$, and we define
\[\begin{aligned} a_0(T) &= a(T; \mathbb{E}_{(pN_0, N_1, N_2)}), \\
a_1(T) &= a(T; \mathbb{E}_{(N_0, pN_1, N_2)}), \\
a_2(T) &= a(T; \mathbb{E}_{(N_0, N_1, pN_2)}). \end{aligned}\]

\begin{lemma}\label{basic-formula-for-ai} In the above notation \[ \begin{aligned} a_0(T) &= \frac{(p^{3k-2}+p^{2k-1}-p^{2k-2}+p^{k+1}-p^k-p+1)a(T) - (p^{2k-1}+p^{k+1}+p^2-p)a(pT) + p^2a(p^2T)}{(p^k-1)(p^{2k-2}-1)}, \\
a_1(T) &= \frac{(-p^{2k-1}- p^{k+1}-p^3+p)a(T) + (p^{2k-1}+p^{k+1}+p^3+p^2-p+p^{-k+4})a(pT) - (p^2 + p^{-k+4})a(p^2T)}{(p^k-1)(p^{2k-2}-1)}, \\
a_2(T) &= \frac{p^3a(T) - (p^3 + p^{-k+4})a(pT) + p^{-k+4}a(p^2T)}{(p^k-1)(p^{2k-2}-1)}. \end{aligned}\] \end{lemma}
\begin{proof}  By Corollary 2.4 we have
\begin{equation}\label{eisenstein-decomposition} \mathbb{E}_{(N_0, N_1, N_2)} = \mathbb{E}_{(pN_0, N_1, N_2)} + \mathbb{E}_{(N_0, pN_1, N_2)} + \mathbb{E}_{(N_0, N_1, pN_2)}, \end{equation}
and comparing the $T$th Fourier coefficient in this gives
\begin{equation}\label{basic-fc-relation1} a(T) = a_0(T) + a_1(T) + a_2(T).\end{equation}
Now apply $U(p)$ to (\ref{eisenstein-decomposition}).  By Proposition \ref{Up-action} we have
\[\begin{aligned} \mathbb{E}_{(N_0, N_1, N_2)} | U(p) &= \mathbb{E}_{(pN_0, N_1, N_2)} \\
&\qquad + (p^{k-1} + 1 - p^{-1}) \mathbb{E}_{(N_0, pN_1, N_2)} \\
&\qquad + (p^{2k-3} + p^{k-1} - p^{k-3} + 1 - p^{-1}) \mathbb{E}_{(N_0, N_1, pN_2)}. \end{aligned}\]
Note that $\mathbb{E}_{(N_0, N_1, N_2)} | U(p)$ makes sense since $\mathbb{E}_{(N_0, N_1, N_2)}$, \textit{a priori} a modular form of level $N$, is also a modular form of level $Np$.  Hence by Proposition \ref{action-on-fc} $a(T; \mathbb{E}_{(N_0, N_1, N_2)}| U(p)) = a(pT; \mathbb{E}_{(N_0, N_1, N_2)})$ and we have
\begin{equation}\label{basic-fc-relation2} \begin{aligned} a(pT) &= a_0(T) \\
&\qquad + (p^{k-1} + 1 - p^{-1})a_1(T) \\
&\qquad + (p^{2k-3} + p^{k-1} - p^{k-3} + 1 - p^{-1})a_2(T). \end{aligned}\end{equation}
Similarly, apply $U_1(p^2)$ to (\ref{eisenstein-decomposition}) we obtain
\begin{equation}\label{basic-fc-relation3} \begin{aligned} a(T; \mathbb{E}_{(N_0, N_1, N_2)} | U_1(p^2)) &= (p+1)a_0(T) \\
&\qquad + (p^{2k-2} + p^{k-1} - p^{k-2} + p + 1 - p^{-1})a_1(T) \\
&\qquad + (p^{2k-2} + p^{2k-3} + p^{k-1} - p^{k-3} + p + 1 - p^{-1} - p^{-2})a_2(T). \end{aligned}\end{equation}
Solving (\ref{basic-fc-relation1}), (\ref{basic-fc-relation2}) and (\ref{basic-fc-relation3}) simultaneously we obtain
\begin{subequations}\label{a-penultimate}
\begin{alignat}{2}
a_0(T) &= \frac{(p^{3k-2} - p^{2k-2} + p^{k+1} - p^k - p + 1)a(T) + (p^k + p)a(pT) -p^ka(T; \mathbb{E}_{(N_0, N_1, N_2)} | U_1(p^2))}{(p^k-1)(p^{2k-2}-1)}, \label{a0-penultimate} \\
a_1(T) &= \frac{(-p^3+p)a(T) -(p^{k+1}+p^k+p^2+p)a(pT) + (p^k + p^2)a(T; \mathbb{E}_{(N_0, N_1, N_2)} | U_1(p^2))}{(p^k-1)(p^{2k-2}-1)}, \label{a1-penultimate} \\
a_2(T) &= \frac{(- p^{k+1} + p^3)a(T) + (p^{k+1} + p^2)a(pT) -p^2a(T; \mathbb{E}_{(N_0, N_1, N_2)} | U_1(p^2))}{(p^k-1)(p^{2k-2}-1)}. \label{a2-penultimate}
\end{alignat}
\end{subequations}

Comparing Fourier expansions at $pT$ in Proposition \ref{Tp-action} we have
\[a(pT; \mathbb{E}_{(N_0, N_1, N_2)}|T(p)) = (p^{2k-3} + p^{k-1} + p^{k-2} + 1)a(pT)\]
On the other hand, by Proposition \ref{action-on-fc},
\[\begin{aligned} a(pT; \mathbb{E}_{(N_0, N_1, N_2)}|T(p)) &= a(p^2T) + p^{k-2} \left(\sum_{\alpha \bmod p} a\left(T\left[\left(\begin{matrix}1 & 0 \\ \alpha & p \end{matrix}\right)\right]\right) + a\left(T\left[\left(\begin{matrix}p & 0 \\ 0 & 1 \end{matrix}\right)\right]\right)\right) + p^{2k-3}a\left(T\right)\\
&= a(p^2T) + p^{k-2}a(T; \mathbb{E}_{(N_0, N_1, N_2)} | U_1(p^2)) + p^{2k-3}a(T).\end{aligned}\]
Hence
\[a(T; \mathbb{E}_{(N_0, N_1, N_2)} | U_1(p^2)) = -p^{-k+2}a(p^2T) + (p^{k-1} + p + 1 + p^{-k+2})a(pT) - p^{k-1}a(T),\]  
and substituting this in to (\ref{a0-penultimate}), (\ref{a1-penultimate}), (\ref{a2-penultimate}) we obtain the lemma.  \end{proof}

\textbf{Formul\ae\ for the Fourier coefficients.}   Note that, given the Fourier coefficients $a(T) = a(T; \mathbb{E}_{(N_0, N_1, N_2)})$, Lemma \ref{basic-formula-for-ai} provides a formula for the Fourier coefficients $a_0(T)$, $a_1(T)$, and $a_2(T)$.  As these are written they are, of course, unsatisfactory; we will now present them in a more familiar form.\\

Before proceeding let us recall the formula from \cite{EichlerZagier1985} for the Fourier coefficients of the level 1 Siegel-Eisenstein series $\mathbb{E}$ of degree $2$ at a positive definite matrix $T$.  This formula is
\begin{equation}\label{level1-fc} a(T; \mathbb{E}) = \frac{2}{\zeta(1-k)\zeta(3-2k)} \sum_{d \mid e(T)} d^{k-1}H\left(\frac{\Delta(T)}{d^2}\right)\end{equation}
where $e\left(\left(\begin{smallmatrix} a & b/2 \\ b/2 & c \end{smallmatrix}\right)\right) = \gcd(a, b, c)$, and $H$ denotes the function defined by Cohen in \cite{Cohen1975} (with first parameter in the notation of \cite{Cohen1975} set equal to $k-1$): writing a positive integer $M$ with $M \equiv 0, -1 \bmod 4$ as $M = -Df^2$ where $D < 0$ is fundamental discriminant the function is
\[H(M) = L(2-k, \chi_D) \sum_{g \mid f} \mu(g) \chi_D(g) g^{k-2} \sum_{h \mid (f/g)} h^{2k-3},\]
where $\chi_D$ is the character associated with the extension $\mathbb{Q}(\sqrt{D})$.  Now let $N$ be any (squarefree) positive integer and let $\mathbf{1}_N$ denote the trivial character modulo $N$.  For $M = -Df^2$ as above we define 
\[H_N(M) = L(2-k, \chi_D) \sum_{g \mid f} \mathbf{1}_N(g) \mu(g) \chi_D(g) g^{k-2} \sum_{h \mid (f/g)} \mathbf{1}_N(h) h^{2k-3}.\]
Note that $H_1 = H$.  Let us also remark that \cite{EichlerZagier1985} provides a formula for the Fourier coefficient $a(T; \mathbb{E})$ when $T$ is singular, namely
\begin{equation}\label{level1-fc-singular} a\left(\left(\begin{matrix} n & 0 \\ 0 & 0 \end{matrix}\right);\: \mathbb{E}\right) = \begin{cases} \frac{2}{\zeta(1-k)} \sum_{d \mid n} d^{k-1} & \text{if }n > 0, \\ 1 & \text{if }n = 0. \end{cases}\end{equation}
This is of course an illustration of how the Fourier coefficients of Eisenstein series of degree $n$ on singular matrices are given by those of Eisenstein series of degree $n-1$.

\begin{lemma}\label{lem:props-of-HN}  Let $N$ be a squarefree positive integer, $p$ a prime not dividing $N$, $k$ a positive integer, $M$ a positive integer with $M \equiv 0, -1 \bmod 4$.  Write $M = -Df^2$ as above, then
\[ H_{Np}(M)C_{p, D}(\ord_p(f)) = H_N(M)\]
where
\[C_{p, D}(v) = \sum_{j=0}^{v} p^{j(2k-3)} - \chi_D(p) p^{k-2} \sum_{j=0}^{v-1} p^{j(2k-3)}.\]
Moreover, writing $p^2M = -D(pf)^2$, we also have
\[H_{Np}(p^2M) = H_{Np}(M).\]\end{lemma}
\begin{proof}  From the definition we have
\[\begin{aligned}H_N(M) &= L(2-k, \chi_D) \left[\sum_{\substack{g \mid f \\ \ord_p(g)=0}} \mathbf{1}_N(g) \mu(g) \chi_D(g) g^{k-2} \sum_{h \mid (f/g)} \mathbf{1}_N(h) h^{2k-3} \right.\\
&\left.\qquad - \chi_D(p)p^{k-2} \sum_{\substack{g \mid f \\ \ord_p(g)=0}} \mathbf{1}_N(g) \mu(g) \chi_D(g) g^{k-2} \sum_{h \mid (f/(pg))} \mathbf{1}_N(h) h^{2k-3} \right] \\
&= L(2-k, \chi_D) \sum_{g \mid f} \mathbf{1}_{Np}(g) \mu(g) \chi_D(g) g^{k-2} \sum_{\substack{h \mid (f/g) \\ \ord_p(h)=0}} \mathbf{1}_N(h)h^{2k-3} \\
&\qquad\times \left[\sum_{j=0}^{\ord_p(f)} p^{j(2k-3)} - \chi_D(p) p^{k-2} \sum_{j=0}^{\ord_p(f)-1} p^{j(2k-3)}\right] \\
&= L(2-k, \chi_D) \sum_{g \mid f} \mathbf{1}_{Np}(g) \mu(g) g^{k-2} \sum_{h \mid (f/g)} \mathbf{1}_{Np}(h)h^{2k-3}C_{p, D}(\ord_p(f)).\end{aligned}\]
The second claimed equality follows immediately from the definition of $H_{Np}$.  \end{proof}

\begin{theorem}\label{fc-of-eisenstein-series}  Let $T = \left(\begin{smallmatrix} m & r/2 \\ r/2 & n \end{smallmatrix}\right)$ be positive semidefinite.  Let $\Delta = 4mn-r^2$ and $e = \gcd(m, n, r)$.  Write $-\Delta = Df^2$ where $D$ is a fundamental discriminant.  Let $\chi_D$ denote the character $\left(\frac{D}{\cdot}\right)$, and let $\mathbf{1}_N(\cdot)$ be the trivial character modulo $N$.  Then
\begin{enumerate}  \item  at $T = \left(\begin{smallmatrix} 0 & 0 \\ 0 & 0 \end{smallmatrix}\right)$, the Fourier coefficients are as follows:
\[a\left(\left(\begin{smallmatrix} 0 & 0 \\ 0 & 0 \end{smallmatrix}\right); \mathbb{E}_{(N_0, N_1, N_2)}\right) = \begin{cases} 1 & \text{if }(N_0, N_1, N_2) = (N, 1, 1), \\ 0 & \text{otherwise.} \end{cases}\]

\item  for $T \neq \left(\begin{smallmatrix} 0 & 0 \\ 0 & 0 \end{smallmatrix}\right)$ but $\Delta=0$, the Fourier coefficients are
\[a(T; \mathbb{E}_{(N_0, N_1, N_2)}) = \Upsilon(T; N_0, N_1, N_2)\frac{2}{\zeta(1-k)}\sum_{d \mid e} \mathbf{1}_N(d) d^{k-1},\]
where $\Upsilon(T; N_0, N_1, N_2) = \prod_{i} \prod_{p \mid N_i} \upsilon_i(p, \ord_p(e))$ with
\[\begin{aligned}  \upsilon_0(p, u_p) &= \frac{p^{(u_p+1)(k-1)}-1}{p^{k-1}-1} - \frac{p^{(u_p+1)(k-1)}p}{p^{k}-1}, \\
\upsilon_1(p, u_p) &= \frac{p^{(u_p+1)(k-1)}p}{p^k-1}, \\
\upsilon_2(p, u_p) &= 0.\end{aligned}\] 

\item for $T>0$, the Fourier coefficients are
\[a(T; \mathbb{E}_{(N_0, N_1, N_2)}) = \Psi(T; N_0, N_1, N_2) \frac{2}{\zeta(1-k)\zeta(3-2k)} \sum_{d \mid e} \mathbf{1}_N(d) d^{k-1} H_N\left(\frac{\Delta}{d^2}\right),\] 
where $\Psi(T; N_0, N_1, N_2) = \prod_i \prod_{p \mid N_i} \psi_i(p, \ord_p(e), \ord_p(f))$ with
\[\begin{aligned} \psi_0(p, u_p, v_p) &= (p^{2k-3} - \chi_D(p)p^{k-2})\left[p^{v_p(2k-3)}\left(\frac{p^{k-2}(p-1)}{(p^{2k-3} - 1)(p^{2k-2}-1)(p^{k-2}-1)}\right) \right.\\
&\left.\qquad\qquad - p^{(v_p-u_p)(2k-3)}p^{u_p(k-1)} \left(\frac{p-1}{(p^{2k-3}-1)(p^k-1)(p^{k-2}-1)}\right) \right]\\
&\qquad+ (\chi_D(p)p^{k-2} - 1)\left[p^{u_p(k-1)}\left(\frac{p^{k-1}(p-1)}{(p^{2k-3}-1)(p^k-1)(p^{k-1}-1)}\right) \right.\\
&\qquad\qquad\left. - \frac{1}{(p^{2k-3}-1)(p^{k-1}-1)}  \right],\\
\psi_1(p, u_p, v_p) &= (p^{2k-3} - \chi_D(p)p^{k-2})\left[p^{v_p(2k-3)}\left(\frac{p^{k-1}(p^{2} - 1)}{(p^{2k-2}-1)(p^k-1)(p^{k-2}-1)}\right)\right.\\
&\qquad\qquad\left. - p^{(v_p-u_p)(2k-3)}p^{u_p(k-1)}\left(\frac{p(p^{k-1}-1)}{(p^{2k-3}-1)(p^k-1)(p^{k-2}-1)}\right) \right]\\
&\qquad+ (\chi_D(p)p^{k-2} - 1)p^{u_p(k-1)}\frac{p^k}{(p^{2k-3}-1)(p^k-1)},\\
\psi_2(p, u_p, v_p) &= (p^{2k-3} - \chi_D(p)p^{k-2})p^{v_p(2k-3)}\frac{p^{k+1}}{(p^{2k-2}-1)(p^k-1)}.\end{aligned}\]

\end{enumerate} \end{theorem}

\begin{proof}  Arguing by induction on the number of divisors of $N$, using \ref{basic-formula-for-ai} and the base case (\ref{level1-fc-singular}), one obtains 1. and 2.  These Fourier coefficients could also be obtained by considering the cusp of support of $\Phi(\mathbb{E}_{(N_0, N_1, N_2})$ to identify this as a degree $1$ Eisenstein series.  The more interesting case is that of 3.  Here we will again proceed by induction on the number of prime divisor of $N$, but now the calculations are more technical.  The base case is the formula (\ref{level1-fc}) from \cite{EichlerZagier1985} (we have the usual convention that any product indexed by the empty set is equal to $1$).\\

Now suppose we have a multiplicative partition $(N_0, N_1, N_2)$ of the square-free integer $N$, $p$ is a prime not dividing $N$, and the coefficients $a(T) = a(T; \mathbb{E}_{(N_0, N_1, N_2)})$ are as stated in the theorem.  To ease notation we shall write $u_p = \ord_p(e)$, $v_p = \ord_p(f)$.  Using both parts of Lemma \ref{lem:props-of-HN} we can write
\begin{equation}\label{introduce-p-raising-coefficient}\begin{aligned} a(T) &= \Psi(T; N_0, N_1, N_2) \frac{2}{\zeta(1-k)\zeta(3-2k)} \sum_{j=0}^{u_p} \sum_{d \mid e} \mathbf{1}_{Np}(d)d^{k-1} p^{j(k-1)} H_N\left(\frac{\Delta}{p^{2j}d^2}\right) \\
&= \Psi(T; N_0, N_1, N_2) \frac{2}{\zeta(1-k)\zeta(3-2k)} \sum_{j=0}^{u_p} \sum_{d \mid e} \mathbf{1}_{Np}(d)d^{k-1} p^{j(k-1)} H_{Np}\left(\frac{\Delta}{p^{2j}d^2}\right)C_{p, D}\left(v_p-j\right)\\
&= \frac{2}{\zeta(1-k)\zeta(3-2k)} \sum_{d \mid e} \mathbf{1}_{Np}(d) d^{k-1} H_{Np}\left(\frac{\Delta(T)}{d^2}\right)\left[\Psi(T; N_0, N_1, N_2) \sum_{j=0}^{u_p} p^{j(k-1)} C_{p, D}\left(v_p-j\right)\right].\end{aligned}\end{equation}
Similarly,
\[\begin{aligned} a(pT) &= \frac{2}{\zeta(1-k)\zeta(3-2k)} \sum_{d \mid e} \mathbf{1}_{Np}(d) d^{k-1} H_{Np}\left(\frac{\Delta}{d^2}\right) \\
&\qquad\times \left[\Psi(pT; N_0, N_1, N_2) \sum_{j=0}^{u_p+1} p^{j(k-1)}C_{p, D}\left(v_p + 1 - j\right)\right],\\
a(p^2T) &= \frac{2}{\zeta(1-k)\zeta(3-2k)} \sum_{d \mid e} \mathbf{1}_{Np}(d) d^{k-1} H_{Np}\left(\frac{\Delta}{d^2}\right)\\
&\qquad\times\left[\Psi(p^2T; N_0, N_1, N_2) \sum_{j=0}^{u_p+2} p^{j(k-1)}C_{p, D}\left(v_p + 2 - j\right)\right].\end{aligned}\]
If $\sum_{d \mid e} \mathbf{1}_{Np}(d)d^{k-1} H_{Np}(\Delta/d^2) = 0$ then the above formul\ae\ and Lemma \ref{basic-formula-for-ai} give the result, so we may assume not.  Then, again by Lemma \ref{basic-formula-for-ai}, we have
\[\begin{aligned} &(p^{k}-1)(p^{2k-2}-1)\Psi(T; pN_0, N_1, N_2) = \\
&\qquad\qquad (p^{3k-2} + p^{2k-1} - p^{2k-2} + p^{k+1} - p^k - p + 1)\left[\Psi(T; N_0, N_1, N_2)\sum_{j=0}^{u_p} p^{j(k-1)} C_{p, D}\left(v_p - j\right)\right] \\
&\qquad\qquad+ (-p^{2k-1} - p^{k+1} - p^2 + p)\left[\Psi(pT; N_0, N_1, N_2) \sum_{j=0}^{\ord_p(e(T))+1} p^{j(k-1)}C_{p, D}\left(v_p + 1 - j\right)\right]\\
&\qquad\qquad+ p^2\left[\Psi(p^2T; N_0, N_1, N_2)\sum_{j=0}^{u_p+2} p^{j(k-1)}C_{p, D}\left(v_p + 2 - j\right)\right]\\
&(p^k-1)(p^{2k-2}-1)\Psi(T; N_0, pN_1, N_2) = \\
&\qquad\qquad(-p^{2k-1} - p^{k+1} - p^3 + p)\left[\Psi(T; N_0, N_1, N_2)\sum_{j=0}^{u_p} p^{j(k-1)} C_{p, D}\left(v_p - j\right)\right] \\
&\qquad\qquad+ (p^{2k-1} + p^{k+1} + p^3 + p^2 - p + p^{-k+4}) \left[\Psi(pT; N_0, N_1, N_2)\sum_{j=0}^{u_p+1} p^{j(k-1)}C_{p, D}\left(v_p + 1 - j\right)\right]\\
&\qquad\qquad+ (-p^2 - p^{-k+4})\left[\Psi(p^2T; N_0, N_1, N_2)\sum_{j=0}^{u_p+2} p^{j(k-1)}C_{p, D}\left(v_p + 2 - j\right)\right],\\
&(p^k-1)(p^{2k-2}-1)\Psi(T; N_0, N_1, pN_2) = \\
&\qquad\qquad p^3\left[\Psi(T; N_0, N_1, N_2) \sum_{j=0}^{u_p} p^{j(k-1)} C_{p, D}\left(v_p - j\right)\right] \\
&\qquad\qquad+ (-p^3 - p^{-k+4})\left[\Psi(pT; N_0, N_1, N_2)\sum_{j=0}^{u_p+1} p^{j(k-1)}C_{p, D}\left(v_p + 1 - j\right)\right]\\
&\qquad\qquad+ p^{-k+4}\left[\Psi(p^2T; N_0, N_1, N_2) \sum_{j=0}^{u_p+2} p^{j(k-1)}C_{p, D}\left(v_p + 2 - j\right)\right].\end{aligned}\]
We aim to find a solution to these equations subject to the initial condition $\Psi(T; 1, 1, 1) = 1$.  From this initial condition and the right and side of the above formul\ae\ we see that $\Psi(T; N_0, N_1, N_2)$ only depends on $T$ via the underlying fundamental discriminant $D$ and the local quantities $u_q = \ord_q(e(T))$ and $v_q = \ord_q(f)$ at primes $q \mid N_0N_1N_2$.  In particular if $p$ is a prime not dividing $N_0 N_1 N_2$ then $\Psi(pT; N_0, N_1, N_2) = \Psi(T; N_0, N_1, N_2)$.  Thus we can remove $\Psi(T; N_0, N_1, N_2)$ as a common factor from all terms on the right hand side of the above system of equations, which then simplify to
\[\begin{aligned} &(p^k-1)(p^{2k-2}-1)\Psi(T; pN_0, N_1, N_2) = \\
&\qquad\left[(p^{2k-1} - p^{2k-2} - p + 1)\sum_{j=0}^{u_p}p^{j(k-1)} C_{p, D}(v_p-j) -(p^{2k-1} + p^2 - p)C_{p, D}(v_p + 1) + p^2 C_{p, D}(v_p + 2)\right] \\
&\qquad\times \Psi(T; N_0, N_1, N_2), \\
&(p^k-1)(p^{2k-2}-1)\Psi(T; N_0, pN_1, N_2) = \\
&\qquad\left[(p^{3k-2} - p^{2k-1} - p^{k} + p)\sum_{j=0}^{u_p} p^{j(k-1)} C_{p, D}(v_p-j) + (p^{2k-1} + p^2 - p + p^{-k+4})C_{p, D}(v_p + 1)\right.\\
&\qquad \qquad\left. - (p^2+p^{-k+4})C_{p, D}(v_p + 2)\right]\Psi(T; N_0, N_1, N_2), \\
&(p^k-1)(p^{2k-2}-1)\Psi(T; N_0, N_1, pN_2) = \\
&\qquad \left[-p^{-k+4}C_{p, D}(v_p + 1) + p^{-k+4}C_{p, D}(v_p + 2)\right] \\
&\qquad \times\Psi(T; N_0, N_1, N_2).\end{aligned}\]
In the third of these we note that 
\begin{equation} \label{expand-C} C_{p, D}(v_p+2) = C_{p, D}(v_p + 1) + p^{(v_p+2)(2k-3)} - \chi_D(p)p^{k-2}p^{(v_p+1)(2k-3)},\end{equation} 
thus 
\[\begin{aligned} (p^k-1)(p^{2k-2}-1)\Psi(T; N_0, N_1, pN_2) &= p^{k+1}\left[p^{(v_p+1)(2k-3)} - \chi_D(p)p^{k-2}p^{v_p(2k-3)}\right]\Psi(T; N_0, N_1, N_2)\\
&= p^{v_p(2k-3)} \left[p^{3k-2} - \chi_D(p)p^{2k-1}\right]\Psi(T; N_0, N_1, N_2).\end{aligned}\]
This gives the formula for $\psi_2$ stated in the theorem.  For $\psi_1$ we note that $p^{3k-2} - p^{2k-1} - p^k + p = (p^{k-1}-1)(p^{2k-1} - p)$, so
\[\begin{aligned}&(p^{3k-2} - p^{2k-1} - p^k + p)\sum_{j=0}^{u_p} p^{j(k-1)} C_{p, D}(v_p - j) \\
&\qquad= (p^{2k-1} - p)\left[\sum_{j=0}^{u_p}p^{(j+1)(k-1)} C_{p, D}(v_p - j) - \sum_{j=0}^{u_p}p^{j(k-1)} C_{p, D}(v-j)\right] \\
&\qquad= (p^{2k-1} - p)\left[p^{(u_p+1)(k-1)}C_{p, D}(v_p-u_p) - C_{p, D}(v_p)  \right.\\
&\qquad\qquad\left. + \sum_{j=0}^{u_p-1} p^{(j+1)(k-1)}\left[C_{p, D}(v_p-j) - C_{p, D}(v_p-j-1)\right]\right] \\
&\qquad = (p^{2k-1} - p) \left[p^{(u_p+1)(k-1)}C_{p, D}(v_p-u_p) - C_{p, D}(v_p)  \right.\\
&\qquad\qquad\left. + \sum_{j=0}^{u_p-1} p^{(j+1)(k-1)}\left[p^{(v_p-j)(2k-3)} - \chi_D(p)p^{k-2}p^{(v_p-j-1)(2k-3)}\right]\right]\end{aligned}\]
Also, expanding as with (\ref{expand-C}) we have
\[\begin{aligned}&(p^{2k-1} + p^2 - p + p^{-k+4})C_{p, D}(v_p + 1) - (p^2 + p^{-k+4})C_{p, D}(v_p+2) \\
&\qquad = (p^{2k-1} - p)C_{p, D}(v) - (p^{k+1} + p)p^{v_p(2k-3)}(p^{2k-3} - \chi_D(p)p^{k-2})\end{aligned}\]
Combining these in to the formula for $\Psi(T; N_0, pN_1, N_2)$  we obtain
\[\begin{aligned}&(p^k-1)(p^{2k-2}-1)\Psi(T; N_0, pN_1, N_2) \\
&=\left\{ (p^{2k-1}-p) \left[p^{(u_p + 1)(k-1)}C_{p, D}(v_p - u_p) + \sum_{j=0}^{u_p-1} p^{(j+1)(k-1)}\left[p^{(v_p-j)(2k-3)} - \chi_D(p)p^{k-2}p^{(v_p-j-1)(2k-3)}\right]\right] \right.\\
&\left.\qquad - (p^{k+1} + p)p^{v_p(2k-3)}\left(p^{2k-3} - \chi_D(p)p^{k-2}\right)\right\}\Psi(T; N_0, N_1, N_2)\\
&=\left\{p^{u_p(k-1)}p^{k-1}(p^{2k-1} - p) \right.\\
&\left.\qquad+ \left[(p^{2k-1} - p)\left(p^{k-1}\frac{p^{(v_p-u_p)(2k-3)}p^{u_p(k-1)} -p^{u_p(k-1)}}{p^{2k-3} - 1} + \frac{p^{v_p(2k-3)}-p^{(v_p-u_p)(2k-3)}p^{u_p(k-1)}}{p^{k-2}-1}\right) \right.\right.\\
&\left.\qquad\qquad- (p^{k+1} + p)p^{v_p(2k-3)}\left. \right](p^{2k-3} - \chi_D(p)p^{k-2})\right\}\Psi(T; N_0, N_1, N_2) \\
&=\left\{(p^{2k-3} - \chi_D(p)p^{k-2})\left[p^{v_p(2k-3)}\left(\frac{p^{k-1}(p^{2} - 1)}{p^{k-2}-1}\right) - p^{(v_p-u_p)(2k-3)}p^{u_p(k-1)}\left(\frac{p(p^{2k-2}-1)(p^{k-1}-1)}{(p^{2k-3}-1)(p^{k-2}-1)}\right) \right]\right.\\
&\left.\qquad+ (\chi_D(p)p^{k-2} - 1)p^{u_p(k-1)}\frac{p^k(p^{2k-2}-1)}{p^{2k-3}-1}\right\}\Psi(T; N_0, N_1, N_2).\end{aligned}\]
This gives the formula for $\psi_1$ stated in the theorem.  One can argue in a similar fashion to derive the formula for $\psi_0$, but given that we have found these formul\ae\ to $\psi_1$ and $\psi_2$ it is less painful to instead argue from the observation that by (\ref{introduce-p-raising-coefficient}) we have
\[\Psi(T; pN_0, N_1, N_2) + \Psi(T; N_0, pN_1, N_2) + \Psi(T; N_0, N_1, pN_2) = \Psi(T; N_0, N_1, N_2) \sum_{j=0}^{u_p} p^{j(k-1)} C_{p, D}(v_p-j)\]
so we can obtain $\psi_0$ by evaluating the sum on the right hand side.  But this is easily done, namely
\[\begin{aligned}\sum_{j=0}^{u_p} p^{j(k-1)} C_{p, D}(v_p - j) &= \sum_{j=0}^{u_p} p^{j(k-1)} \left[1 + (p^{2k-3} - \chi_D(p)p^{k-2}) \sum_{i=0}^{v_p-j-1}p^{i(2k-3)} \right] \\
&= \left(\frac{p^{(u_p+1)(k-1)} - 1}{p^{k-1} - 1}\right) + (p^{2k-3} - \chi_D(p)p^{k-2})\sum_{j=0}^{u_p} p^{j(k-1)} \left(\frac{p^{(v_p-j)(2k-3)}-1}{p^{2k-3}-1}\right)\\
&= (p^{2k-3} - \chi_D(p)p^{k-2})\left[p^{v_p(2k-3)} \left(\frac{p^{k-2}}{(p^{2k-3}-1)(p^{k-2}-1)}\right) \right.\\
&\left.\qquad\qquad - p^{(v_p-u_p)(2k-3)}p^{u_p(k-1)}\left(\frac{1}{(p^{2k-3}-1)(p^{k-2}-1)}\right)\right] \\
&\qquad + (\chi_D(p)p^{k-2} - 1)\left[p^{u_p(k-1)} \left(\frac{p^{k-1}}{(p^{2k-3}-1)(p^{k-1}-1)}\right) \right.\\
&\left.\qquad\qquad -\frac{1}{(p^{2k-3}-1)(p^{k-1}-1)}\right].\end{aligned}\]
Subtracting $\psi_1 + \psi_2$ from this we obtain the formula for $\psi_0$ stated in the theorem.

\end{proof}

\section{Applications to representation numbers of quadratic forms}\label{sctn:representation-numbers}

Let $L$ be a lattice in $\mathbb{Z}^{2k}$ endowed with a quadratic form $Q : L \to \mathbb{Z}$.  Then $Q$ defines a symmetric bilinear form on $L$ by the formula
\[B(x, y) = Q(x + y) - Q(x) - Q(y)\]
which is integer valued and moreoever satisfies $B(x, x) \in 2\mathbb{Z}$ for all $x \in L$.  Conversely given such a bilinear form $B$ we can define a quadratic form by the rule $Q(x) = \frac{1}{2}B(x, x)$.  This sets up a bijection, so that specifying a $\mathbb{Z}$-valued quadratic form $Q$ is equivalent to specifying a $\mathbb{Z}$-valued bilinear form $B$ such that $B(x, x) \in 2\mathbb{Z}$ for all $x \in L$.  We shall refer to a lattice endowed with either of these equivalent structures as an even lattice.  Picking a basis $(e_1,...,e_{2k})$ for $L$ we can map $B$ to the matrix $B(e_i, e_j)$, which we call a Gram matrix for $L$.  A Gram matrix then has integer entries and is even ones on the diagonal, we call such a matrix \textit{even integral}, \\

For such an even lattice $L$ we form the degree $n$ theta series by
\[\begin{aligned}\theta_L^{(n)}(Z) &= \sum_{X \in \mathbb{Z}^{2k, n}} e^{\pi i\tr({}^tX S X Z)} \\
&= \sum_{\substack{T \geq 0 \\ T\text{ even integral}}} r_S(T) e^{\pi i\tr(TZ)},\end{aligned}\]
where $S$ is any Gram matrix for $L$, and $r_S(T) = \#\{X \in \mathbb{Z}^{2k, n};\; {}^tXSX = T\}$ is the number of representations of the $n$-variable quadratic form $T$ by $S$.  This in the form of the Fourier expansion of a Siegel modular form (c.f. (\ref{fourier-expansion})).  It is well-known that $\theta_L^{(n)}$ is indeed a Siegel modular form, namely $\theta_L^{(n)} \in \mathcal{M}_k^{(n)}(N, \chi)$ where the level $N$ is the level of the the lattice $L$ (equivalently the smallest integer $N$ such that $NS^{-1}$ is an even integral matrix) and the character is
\[\chi = \left(\frac{(-1)^k \det(S)}{\cdot}\right).\]
Note that if $\det(S)$ is a (global) square then this character is trivial.  Let us remark that both the level and character are genus-invariants of the quadratic form.  \\  

With $L$ as above we write $\theta_{\gen(L)}^{(n)}(Z)$ for the genus theta series of $L$.  This is formed as follows: let $L = L_1, L_2, ..., L_h$ be the inequivalent lattices in the genus of $L$, write $O(L_i)$ for the size of the isometry group of $L_i$, $w = \sum_{i=1}^h \frac{1}{O(L_i)}$, and put
\[\theta_{\gen(L)}^{(n)}(Z) = \frac{1}{w} \sum_{i=1}^h \frac{1}{O(L_i)} \theta_{L_i}.\]
Let $S=S_1, S_2, ..., S_h$ be Gram matrices for $L=L_1, L_2, ..., L_h$ respectively.  Then
\[\theta_{\gen(L)}^{(n)}(Z) = \sum_{\substack{T \geq 0 \\ T\text{ even integral}}} r_{\gen(S)}(T) e^{\pi i\tr(TZ)},\]
where 
\[r_{\gen(S)}(T) = \frac{1}{w} \sum_{i=1}^h \frac{1}{O(S_i)} r_{S_i}(T)\]
measures the average number of representations of $T$ by the genus of $S$.  In this section we consider the problem of computing the average representation numbers $r_{\gen(S)}(T)$.  The key to doing this is Siegel's Hauptsatz, which says that $\theta_{\gen(L)}^{(n)}$ lies in the subspace of $\mathcal{M}_k^{(n)}(N, \chi)$ spanned by Eisenstein series.  Siegel went on to expresses the Fourier coefficients $r_{\gen(L)}(T)$ as a product of $p$-adic densities of solutions to the representation problem.  Here we will give a very explicit formula for these representations numbers (in our special case) in terms of the Fourier cofficients of Siegel-Eisenstein series, under the assumption that the level is squarefree and the character is trivial.\\

Hence suppose that $N$ is squarefree.  It is well-known that the $(n-1)$-cusps of the Satake compactification $S(\Gamma_0^{(n)}(N) \backslash \mathbb{H}_n)$ of the complex analytic space $\Gamma_0^{(n)}(N) \backslash \mathbb{H}_n$ are in bijective correspondence with positive divisors of $N$ (see \cite{Dickson2013} for a description of the full cuspidal configuration of the Satake compactification in this case).  Specifically, one may use the following system of representatives: for each $p \mid N$, fix a matrix $\gamma_p$ satisfying the conditions
\[\gamma_p \equiv \begin{cases} \left(\begin{smallmatrix} 0_n & -1_n \\ 1_n & 0_n\end{smallmatrix}\right) & \mod p, \\ \left(\begin{smallmatrix} 1_n & 0_n \\ 0_n & 1_n\end{smallmatrix}\right) & \mod q\text{ for all }q\mid N, q \neq p.\end{cases}\]
For $d \mid N$, set $\gamma_d = \prod_{p \mid d} \gamma_p$ (and $\gamma_1 = 1_{2n}$).  Write $\Phi$ for the Siegel lowering operator
\[\Phi(f)(Z) = \lim_{\lambda \rightarrow \infty} f\left(\begin{matrix} Z' & 0 \\ 0 & i\lambda\end{matrix}\right),\]
where $Z \in \mathbb{H}_n$ and $Z' \in \mathbb{H}_{n-1}$.  The Siegel lowering operator maps Siegel modular forms of degree $n$ to Siegel modular forms of degree $n-1$, and geometrically corresponds to restricting $f$ to a particular $(n-1)$-cusps.  With $\gamma_d$ as above, the function $f \mapsto \Phi(f | \gamma_d)$ corresponds to restricting $f$ to the cusp represented by $\gamma_d$.

\begin{proposition}\label{value-of-theta-at-cusp}  Let $L$ be an even lattice and $S$ a Gram matrix of $L$.  Suppose that the level $N$ is squarefree and $\det(S)$ is a square, so that $\theta_L \in \mathcal{M}_k(N)$.  For a prime divisor $p$ of $N$, let $s_p(L)$ be the Hasse invariant of $Q$ on $L \otimes_\mathbb{Z} \mathbb{Z}_p$, normalized as in \cite{Scharlau1985}.  For any divisor $d$ of $N$, let $d_p$ be the highest power of $p$ dividing $\det(S)$.  Then  
\[\Phi (\theta_L^{(n)} |_k \gamma_d) = \prod_{p \mid d} \left(d_p^{-n/2} s_p(L)^n\right) \theta_{L^{\#, d}}^{(n-1)}\]
where
\[L^{\#, d} = L^\# \cap \mathbb{Z}\left[\frac{1}{p};\: p \mid d\right]\]
denotes the lattice dualized at all primes $p \mid d$.
\end{proposition}
\begin{proof}  Let $p$ be any prime divisor of $d$.  Applying \cite{BoechererSchulze-Pillot} Lemma 8.2(a) with $l=n$ (noting that the factor $\gamma_p(d_p)=1$ since $d_p$ is a square) we obtain 
\[\theta_L^{(n)} |_k \gamma_p = d_p^{-n/2} s_p(L)^n \theta_{L^{\#, p}}^{(n)}\]
where
\[L^{\#, p} = L^\# \cap \mathbb{Z}\left[\frac{1}{p}\right]L\]
denotes the lattice dualized only at $p$.  Since $d_p$ and $s_p$ are local to $p$ and $N$ is squarefree we can apply this result now at other primes dividing $d$ to see that
\[\theta_L^{(n)} | \gamma_d = \prod_{p \mid d} \left(d_p^{-n/2} s_p(L)^n\right) \theta_{L^{\#, d}}^{(n)}\]
where $L^{\#, d}$ is as in the statement of the proposition.  Applying the Siegel lowering operator we obtain the result.  \end{proof}

\begin{corollary}\label{crly:average-representation-number-formula}   Let $L$ be an even integral lattice of rank $2k$, $k \geq 4$.  Assume that the level $N$ is squarefree and the transformation character of $\theta_L$ trivial, so that $\theta_L \in \mathcal{M}_k(N)$.  Let $\mathbb{E}_{(N_0, N_1, N_2)}$ be the Eisenstein series in the natural basis.  Then
\[\theta_{\gen(L)} = \sum c(N_0, N_1, N_2) \mathbb{E}_{(N_0, N_1, N_2)} \]
where the sum is over all tuples $(N_0, N_1, N_2)$ of positive integers such that $N_0 N_1 N_2 = N$, and the coefficients are given by
\[c(N_0, N_1, N_2) = \prod_{p \mid N_1} d_p^{-1/2} s_p(L) \prod_{p \mid N_2} d_p^{-1}.\]
In particular if $T \in \mathbb{Q}^{2 \times 2}_{\text{sym}}$ is positive definite and semi-integral then the average representation number $r_{\gen(S)}(T)$ is given by
\[r_{\gen(S)}(T) = \sum c(N_0, N_1, N_2)a(T; \mathbb{E}_{(N_0, N_1, N_2)})\]
where the sum and $c(N_0, N_1, N_2)$ are as above, and $a(T; \mathbb{E}_{(N_0, N_1, N_2)})$ is given by Theorem \ref{fc-of-eisenstein-series}.  \end{corollary}
\begin{proof}  By Siegel's Hauptsatz we know that $\theta_{\gen(L)}$ is a linear combination of Eisenstein series.  Now the Eisenstein series comprising our basis are characterised by $\mathbb{E}_{(N_0, N_1, N_2)}$ being the unique weight $k$ and level $N = N_0 N_1 N_2$ Eisenstein series which takes the value $1$ at the cusp corresponding to $(N_0, N_1, N_2)$ and the value $0$ at all others.  One easily checks that this condition is 
\[\Phi \left(\Phi \left(\mathbb{E}_{(N_0, N_1, N_2)} | \gamma_{N_2}\right) | \gamma_{N_1}\right)= 1.\]
Thus to express $\theta_L^{(n)}$ as a linear combination of Eisenstein series it suffices to compute the value of $\theta_L^{(n)}$ at the $0$-cusp $(N_0, N_1, N_2)$.  By Proposition \ref{value-of-theta-at-cusp} we have
\[\Phi\left(\Phi\left(\theta_L^{(n)}| \gamma_{N_2}\right)| \gamma_{N_1}\right) = \prod_{p \mid N_1} d_p^{-1/2} s_p(L) \prod_{p \mid N_2} d_p^{-1},\]
using the fact that $N_1$ and $N_2$ are coprime.\end{proof}

We emphasise that everything in Corollary \ref{crly:average-representation-number-formula} is completely explicit.  To illustrate this we consider the case when the genus of the quadratic form corresponding to $S$ contains only one isomorphism class.  Then the average and exact representation numbers $r_{\gen(S)}(T)$ and $r_S(T)$ are the same object and Corollary \ref{crly:average-representation-number-formula} gives us an exact formula for these.  Now if $S$ is of size $2k$ then it describes a modular form of weight $k$; in order to analyse this with Eisenstien series we require $k$ to be even and at least $4$.  According to the Nebe--Sloane database there are 36 $8$-dimensional lattices which form a single-class genus (and none in higher dimensions divisible by $4$); of these there are 5 which satisfy the condition that the level be squarefree and the transformation character trivial.  As noted in the introduction one of these (which has matrix $S_1$ in the following) is $E_8$, the others are not unimodular but have small prime level.  Explicitly these lattices are the following: we regard a symmetric matrix $S = (s_{ij})$ of size $8$ as being determined by a tuple
\[v(S) = (s_{11}, s_{21}, s_{22}, s_{31}, s_{32}, s_{33}, ..., s_{81}, s_{82}, s_{83}, s_{84}, s_{85}, s_{86}, s_{87}, s_{88}).\] 
Then the Gram matrices $S_i$ for the $8$-dimensional single-genus even lattices of squarefree level and trivial character are determined by
\[\begin{aligned}v(S_1) &= (2, 1, 2, 1, 1, 2, 1, 0, 0, 2, 1, 1, 0, 0, 2, 1, 1, 0, 0, 1, 2, 1, 0, 1, 0, 0, 0, 2, 1, 1, 0, 1, 1, 1, 0, 2),\\
v(S_2) &= (2, -1, 2, 0, -1, 2, 0, 0, -1, 2, 0, 0, 0, -1, 2, 0, 0, -1, 0, 0, 2, 0, 0, 0, 0, 0, 0, 2, 0, 0, 0, 0, 0, 0, 1, 2),\\
v(S_3) &= (2, 1, 2, -1, -1, 2, 1, 1, 0, 2, 1, 1, 0, 1, 2, 1, 1, 0, 1, 1, 2, 1, 1, 0, 1, 1, 1, 2, 1, 1, 0, 1, 1, 1, 1, 2),\\
v(S_4) &= (2, 0, 2, 0, 0, 2, 1, 1, 1, 2, 0, 0, 0, 0, 2, 0, 0, 0, 0, 0, 2, 0, 0, 0, 0, 0, 0, 2, 0, 0, 0, 0, 1, 1, 1, 2),\\
v(S_5) &= (2, 0, 2, 0, 0, 2, 1, -1, 1, 4, 0, 0, 0, 1, 2, 0, 0, 0, -1, 0, 2, 0, 0, 0, -1, 0, 0, 2, 0, 0, 0, 1, 0, 0, 0, 2).\end{aligned}\]
Computing the level of each lattices and applying Proposition \ref{value-of-theta-at-cusp} we obtain the following:
\begin{table}
\begin{center}
\begin{tabular}{|l| l| l|}\hline
Matrix & Level & Number of representations of $T$ \\ \hline
$S_1$ & $1$ & $a(T; \mathbb{E}_{4, (1, 1, 1)})$ \\
$S_2$ & $3$ & $a(T; \mathbb{E}_{4, (3, 1, 1)}) + \left(1/3\right)a(T; \mathbb{E}_{4, (1, 3, 1)}) + \left(1/9\right) a(T; \mathbb{E}_{4, (1, 1, 3)})$ \\
$S_3$ & $2$ & $a(T; \mathbb{E}_{4, (2, 1, 1)}) + \left(1/2\right)a(T; \mathbb{E}_{4, (1, 2, 1)}) + \left(1/4\right) a(T; \mathbb{E}_{4, (1, 1, 2)})$ \\
$S_4$ & $2$ & $a(T; \mathbb{E}_{4, (2, 1, 1)}) + \left(1/4\right)a(T; \mathbb{E}_{4, (1, 2, 1)}) + \left(1/16\right)a(T; \mathbb{E}_{4, (1, 1, 2)})$ \\
$S_5$ & $2$ & $a(T; \mathbb{E}_{4, (2, 1, 1)}) + \left(1/8\right)a(T; \mathbb{E}_{4, (1, 2, 1)}) + \left(1/64\right)a(T; \mathbb{E}_{4, (1, 1, 2)})$ \\ \hline \end{tabular}\caption{Degree two representation numbers of the quadratic forms in eight variables which are the unique form in their genus.}\end{center}\end{table}
With the easily computable formul\ae\ of Theorem \ref{fc-of-eisenstein-series} one can now compute representation numbers of these quadratic forms very quickly on a computer.  Of course the same reasoning applies to allow quick computation of representation numbers of genus-averages of quadratic forms, provided that the level is squarefree and the transformation character is trivial.\\

\textsc{Department of Mathematics, University Walk, Bristol, BS8 1TW.\\
Email address: }\texttt{martin.dickson@bristol.ac.uk}.

\end{document}